\documentclass{article}

\usepackage{amsmath,amssymb,amsthm}
\usepackage{bm}
\usepackage{mathrsfs}
\usepackage[all]{xy}
\usepackage{booktabs}
\usepackage{fullpage}
\usepackage{hyperref}
\usepackage{mathtools}
\hypersetup{
colorlinks=false,
linkcolor=red,
citecolor=green
}

\newcommand{\rl}{\mathbb{R}}
\newcommand{\cx}{\mathbb{C}}

\newcommand{\ai}{\sqrt{-1}}

\newcommand{\inj}{\hookrightarrow}

\newcommand{\cst}{\mathrm{const}}
\newcommand{\kah}{K\"ahler }

\newcommand{\ddbar}{\partial \bar{\partial}}
\newcommand{\tr}{\mathrm{tr}}
\newcommand{\actson}{\curvearrowright}
\newcommand{\prj}{\mathbb{P}}

\newcommand{\XX}{\mathcal{X}}
\newcommand{\LL}{\mathcal{L}}

\theoremstyle{plain}
\newtheorem{theorem}{Theorem}[section]
\newtheorem{lemma}[theorem]{Lemma}
\newtheorem{proposition}[theorem]{Proposition}
\newtheorem{corollary}[theorem]{Corollary}
\theoremstyle{definition}
\newtheorem{definition}[theorem]{Definition}

\theoremstyle{definition}
\newtheorem{remark}[theorem]{Remark}

\begin{document}

\title{Relative stability associated to quantised extremal K\"ahler metrics}
\author{Yoshinori Hashimoto}

\maketitle

\begin{abstract}
We study algebro-geometric consequences of the quantised extremal K\"ahler metrics, introduced in the previous work of the author. We prove that the existence of quantised extremal metrics implies weak relative Chow polystability. As a consequence, we obtain asymptotic weak relative Chow polystability and $K$-semistability of extremal manifolds by using quantised extremal metrics; this gives an alternative proof of the results of Mabuchi and Stoppa--Sz\'ekelyhidi. In proving them, we further provide an explicit local density formula for the equivariant Riemann--Roch theorem.
\end{abstract}

\tableofcontents

\section{Introduction}

Donaldson's work \cite{donproj1} implies that if a polarised \kah manifold $(X,L)$ with discrete automorphism admits a constant scalar curvature \kah (cscK) metric, it admits a sequence of \kah metrics $\left\{ \omega_k \right\}_k$ satisfying $\rho_k (\omega_k) = \text{const}$, where $\rho_k (\omega_k)$ is the Bergman function of $\omega_k$. Combined with the results of Luo \cite{luo} and Zhang \cite{zhang}, this further implies that such $(X,L)$ is asymptotically Chow stable, establishing an important result in \kah geometry connecting the scalar curvature and algebro-geometric stability of $(X,L)$ in the sense of Geometric Invariant Theory (GIT). The reader is referred to the survey \cite{biquard} for more details on this theory.

When the automorphism group is no longer discrete, a generalisation of Donaldson's result was established in \cite{yhpreprint}, widening the scope to include extremal \kah metrics. This was done by considering an equation 
\begin{equation}
	\label{dbgradrom0}
\bar{\partial} \mathrm{grad}^{1,0}_{\omega} \rho_k (\omega) = 0.
\end{equation}
This paper studies consequences of the above equation to GIT stability notions in algebraic geometry.

Fixing a maximal compact subgroup $K$ of the automorphism group (cf.~Remark \ref{rmntmcsbz}), our first application to stability is the following.

\begin{theorem}
	\label{wrelchowstab}
	Suppose that there exists a $K$-invariant Fubini--Study metric $\omega_k \in c_1 (L)$ induced from $X \inj \prj(H^0 (X,L^k)^*)$ which satisfies $\bar{\partial} \mathrm{grad}^{1,0}_{\omega_k} \rho_k (\omega_k) = 0$. Then $(X,L^k)$ is weakly Chow polystable relative to the centre of $K$.
\end{theorem}

We shall see in the proof that the converse does not hold in general; the solvability of (\ref{dbgradrom0}) is strictly stronger than weak relative Chow polystability (cf.~Remark \ref{rembalstab1dir}).

The main result of \cite{yhpreprint} is that (\ref{dbgradrom0}) is solvable for all large enough $k$, if $( X,L )$ admits an extremal metric (cf.~Theorem \ref{mtqekm}). Combining the main result of \cite{yhpreprint} and Theorem \ref{wrelchowstab}, we obtain the following corollary.

\begin{corollary}
	If a polarised \kah manifold $(X,L)$ admits an extremal \kah metric, then it is asymptotically weakly Chow polystable relative to the centre of $K$.
\end{corollary}
This corollary is also a consequence of the works of Mabuchi \cite{mab04ext,mab05,mab09}. A stronger version of the above corollary was recently proved by Mabuchi \cite{mab2016} (see also \cite{seyrel}).

Our second application to stability is the following.

\begin{theorem}
	\label{relkss}
Suppose that there exists an extremal metric $\omega \in c_1 (L)$. Then $(X,L)$ is $K$-semistable relative to the extremal $\cx^*$-action.
\end{theorem}

\begin{remark}
	Recall that the above theorem was first proved by Stoppa and Sz\'ekelyhidi \cite{stosze} by using the lower bound of the Calabi functional, and then by Mabuchi \cite{mab2014} by using a different method. The point of the above statement is that we give another independent, alternative proof by using the equation (\ref{dbgradrom0}).

	The proof (given in \S \ref{pfthmrelkss}) is conceptually similar to the proof of asymptotic Chow stability implying $K$-semistability \cite{rt2}, but will further involve the detailed analysis of the ``weight'' of relative balanced metrics, in which we make \textit{direct} use of the equation (\ref{dbgradrom0}).
\end{remark}

In proving the above Theorem \ref{relkss}, we shall prove the following ``explicit local density formula'' for the equivariant Riemann--Roch theorem in terms of the Bergman function, which could be interesting in its own right.
\begin{theorem} \label{exeploceqrrtit}
Writing $A_k$ for the generator on $H^0(X,L^k)$ of the product test configuration defined by a Hamiltonian vector field $v$ with Hamiltonian $\psi$ with respect to $\omega_h$, we have
\begin{equation*}
	\frac{1}{k } \mathrm{tr} \left( A_k \right) = - \int_X \psi  {\rho}_k (\omega_h) \frac{\omega^n_h}{n!} - \int_X  \frac{1}{ 4 \pi k} \left( d \psi , d  {\rho}_k (\omega_h) \right)_{\omega_h} \frac{\omega^n_h}{n!}.
\end{equation*}
\end{theorem}

Finally, the recent development in the field \cite{mab2016,santip17,seyrel} means that there are now nontrivial relationships among several versions of ``quantised extremal'' metrics, and implications to relative stability. This will be reviewed in \S \ref{conjrelttprevkres}.

\section*{Organisation of the paper}
After recalling the background in \S \ref{bgoqtst}, we introduce (weak) relative Chow polystability and prove Theorem \ref{wrelchowstab} in \S \ref{stextmfdxl}. Theorem \ref{relkss} is proved in \S \ref{rlksspovquant}, where the definition of relative $K$-semistability is also provided. In \S \ref{expfflrrsn} we shall prove Theorem \ref{exeploceqrrtit}, and the last section \S \ref{conjrelttprevkres} is devoted to the review of the works of \cite{yhpreprint,mab2016,santip17,seyrel} from the point of view of relative stability.

\section*{Acknowledgements}
Much of this work was carried out in the framework of the Labex Archim\`ede (ANR-11-LABX-0033) and of the A*MIDEX project (ANR-11-IDEX-0001-02), funded by the ``Investissements d'Avenir" French Government programme managed by the French National Research Agency (ANR). Part of this work was carried out when the author was a PhD student at the Department of Mathematics of the University College London, which he thanks for the financial support; \S \ref{stextmfdxl} forms part of the author's PhD thesis submitted to the University College London.

The author thanks Joel Fine, Julien Keller, Jason Lotay, Yasufumi Nitta, Julius Ross, Shunsuke Saito, and Carl Tipler for helpful discussions.

\section{Background on quantisation} \label{bgoqtst}
We first recall the definition of the Fubini--Study metrics. We shall write in what follows $N = N_k$ for $\mathrm{dim}_{\cx} H^0 (X,L^k)$ and $V$ for $\int_X c_1 (L)^n /n!$.
\begin{definition}
Let $\mathcal{B}_k$ be the space of all positive definite hermitian matrices on $H^0 (X,L^k)$, and $\mathcal{H} (X,L)$ be the space of all positively curved hermitian metrics on $L$. 

The \textbf{Hilbert} map $Hilb : \mathcal{H} (X,L) \to \mathcal{B}_k$ is defined by
\begin{equation*}
Hilb (h) : = \frac{N}{V} \int_X h^k ( , ) \frac{\omega_h^n}{n!} .
\end{equation*}

The \textbf{Fubini--Study} map $FS :  \mathcal{B}_k \to \mathcal{H} (X,L) $ is defined by the equation
\begin{equation}
\sum_{i=1}^N |s_i|^2_{FS(H)^k} =1 \label{defoffseq}
\end{equation}
where $\{ s_i \}$ is an $H$-orthonormal basis for $H^0 (X , L^k)$. We shall write $\omega_{FS(H)}$ or $\omega_H$ for the \kah metric associated to $FS(H)$.
\end{definition}

We also recall the following result concerning the automorphism group of polarised \kah manifolds and its linearisation. This is a well-known consequence of the results presented in \cite{fujiki, kobayashi, lebsim, mfk}. Let $\mathrm{Aut}_0 (X,L)$ be the connected component of the group $\mathrm{Aut} (X,L)$ which consists of automorphisms of $X$ whose action lift to the total space of the line bundle $L$.

\begin{lemma} \label{lemdefofthtosl}
By replacing $L$ by a large tensor power if necessary, we have a unique faithful group representation
\begin{equation*}
\theta : \textup{Aut}_0 (X , L) \to SL(H^0 (X , L^k))
\end{equation*}
for all $k \in \mathbb{N}$, which satisfies 
\begin{equation} \label{liftauttheta}
\theta (f) \circ \iota = \iota \circ f
\end{equation}
for any $f \in \textup{Aut}_0 (X,L)$ and the Kodaira embedding $\iota : X \inj \prj (H^0 (X , L^k)^*)$. 
\end{lemma}

In what follows, we shall replace $L$ by a large tensor power so that the above lemma holds.

\begin{remark} \label{rmntmcsbz}
It is convenient to fix a maximal compact subgroup $K$ of $\mathrm{Aut}_0 (X,L)$ once and for all. If $(X,L)$ admits an extremal metric $\omega$ we shall take $K$ to be the group of isometry of $\omega$, which is possible by a theorem of Calabi \cite{cal2}. We shall also write $Z$ for the centre of $K$.
\end{remark}

We identify $H^0 (X,L^k)$ with $\cx^N$ by fixing a basis $\left\{ s_i \right\}_i$, to have the isomorphism $\prj (H^0 (X,L^k)^*) \cong \prj^{N-1}.$

\begin{definition}
	Defining a standard Euclidean metric on $\cx^N$ which we write as the identity matrix $I$, we define the \textbf{centre of mass} associated to the basis $\underline{s} = \left\{ s_i \right\}_i$ as
\begin{equation*}
	\bar{\mu}_X (\underline{s}) : =  \int_{X}\frac{h^k_{FS} (s_i , s_j)}{\sum_l |s_l|_{FS^k}^2} \frac{k^n \omega^n_{FS}}{n!}  \in \ai \mathfrak{u} (N) ,
\end{equation*}
where $h^k_{FS} = h^k_{FS(I)}$. 
\end{definition}

\begin{remark}  \label{remdefcomitofs}
	Note that the trace of $\bar{\mu}_X (\underline{s})$ is $k^nV$, and that the equation (\ref{defoffseq}) implies that we in fact have $\bar{\mu}_X (\underline{s}) = \int_X h^k_{FS} (s_i , s_j) \frac{k^n \omega_{FS}^n}{n!}$.
\end{remark}

Recall also the following proposition, by noting that $\bar{\mu}_X (\underline{s})$ is invertible since it is positive definite.

\begin{proposition} \emph{(\cite[Proposition 4.5]{yhpreprint})} \label{invmuholo}
There exists $H \in \mathcal{B}_k$ such that $\bar{\partial} \mathrm{grad}^{1,0}_{\omega_{H}} \rho_k (\omega_{H}) = 0$ if and only if there exists a basis $\left\{ s_i \right\}_i$ for $H^0 (X,L^k)$ such that $\left( \bar{\mu}_X (\underline{s}) \right)^{-1}$ generates a holomorphic vector field on $\prj^{N-1}$ that preserves the image $\iota(X)$ of the Kodaira embedding $\iota : X \inj \prj^{N-1}$.
\end{proposition}

\begin{remark} \label{invmuholoxirem}
	Observe that $\left( \bar{\mu}_X (\underline{s}) \right)^{-1}$ generating a holomorphic vector field preserving $\iota(X)$ is equivalent to $\bar{\mu}_X$ satisfying the following equation
	\begin{equation} \label{invmuholoxi}
		\bar{\mu}_X (\underline{s}) = \left( cI + \xi \right)^{-1}
	\end{equation}
	for some $\xi \in \theta_* (\mathfrak{aut}(X,L))$, where $c \in \rl$ is a constant so that the trace of both sides are equal (to $k^nV$).

	Observe further that when (\ref{invmuholoxi}) is satisfied, the matrix $cI + \xi$ is positive definite hermitian since $\bar{\mu}_X$ is.
\end{remark}

When we solve the equation (\ref{dbgradrom0}) for all large enough $k$ in \cite{yhpreprint}, we prove stronger results with more detailed information on the above $\xi$ and $c$. We consider the following functional.


\begin{definition} \label{defmodbalenza}
	The \textbf{modified balancing energy} $\mathcal{Z}^A$ is defined on the space $\mathcal{B}_k$ of all positive definite hermitian matrices on $H^0 (X,L^k)$ as
	\begin{equation*}
		\mathcal{Z}^A (H(t)) = I \circ FS(H(t)) + \frac{k^n V}{N} \mathrm{tr} \left( \left( I + C_A I + \frac{A}{2 \pi k} \right)^{-1} \log H(t) \right)
	\end{equation*}
	where
	\begin{enumerate}
		\item $I (e^{-\phi_t} h_0) = -k^{n+1} \int_X \phi_t \sum_{i=1}^n (\omega_0 - \ai \ddbar \phi_t)^i \wedge \omega_0^{n-i}$, where $h_0$ is an arbitrarily chosen basepoint in $\mathcal{H} (X,L)$,
		\item $C_A \in \rl$ is some constant so that the trace of the derivative $\delta \mathcal{Z}^A$ is zero,
		\item $A$ is an element in $\theta_*(\ai \mathfrak{z})$, where $\mathfrak{z} = \mathrm{Lie} (Z)$ (cf.~Remark \ref{rmntmcsbz}).
	\end{enumerate}
\end{definition}

An important property of $\mathcal{Z}^A$ is that it is geodesically convex on $\mathcal{B}_k$ and its critical point corresponds to the solution to the equation (\ref{invmuholoxi}) with $c = 1 + C_A$ and $\xi = A / 2 \pi k$, since the linearisation $\delta \mathcal{Z}^{A}$ can be written as
\begin{equation} \label{fmlindzmmbct}
	\delta \mathcal{Z}^A (H (t)) = - \bar{\mu}_X (H(t)) + \frac{Vk^n}{N} \left( I + C_A I + \frac{A}{2 \pi k} \right)^{-1}
\end{equation}
by recalling \cite[\S 5.1]{yhpreprint}.


The main results of \cite{yhpreprint} that we need can be summarised as follows.

\begin{theorem} \emph{(\cite[Theorem 1.4, Corollary 4.15, equation (64)]{yhpreprint})}
	\label{mtqekm}
	Suppose $(X,L)$ admits an extremal metric $\omega$. Then for all $l \in \mathbb{N}$ there exists $k_l \in \mathbb{N}$ such that for all $k \ge k_l$ there exists a hermitian matrix $H_k \in \mathcal{B}_k$ and $A_k \in \theta_* (\mathfrak{z})$ such that the following hold:
	\begin{enumerate}
	\item $\delta \mathcal{Z}^{A_k} (H_k) =0$,
	\item $\omega_k := \omega_{H_k}$ satisfies $\bar{\partial} \mathrm{grad}^{1,0}_{\omega_k} \rho_k (\omega_k) = 0$ and $A_k$ is given by 
	\begin{equation*}
		A_k = \frac{V}{N} \theta_*(\mathrm{grad}_{\omega_k} \rho_k (\omega_k)),
	\end{equation*}
	with the operator norm $||A_k||_{op}$ of $A_k$ being bounded uniformly of $k$ and $C_{A_k} = O(k^{-1})$,
\item $\omega_k$ is $K$-invariant and $\omega_k \to \omega$ in $C^l$.
	\end{enumerate}
\end{theorem}

\section{Relative Chow stability and related concepts} \label{stextmfdxl}

\subsection{Chow stability} \label{stextmfdxlchs}
This is a review of the classical theory, and we refer the reader to \S 1.16 of Mumford's paper \cite{mumford77} and \S 2 of Futaki's survey \cite{fut11} for the details on the materials presented here. Consider a polarised \kah manifold $(X,L)$ with $\dim_{\cx} X=n$ and degree $d_k : = \int_X c_1(L^k)^n$, and the Kodaira embedding $\iota : X \inj \prj ( H^0 (X,L^k)^*)$. Writing $V_k:= H^0(X,L^k)$, observe that $n+1$ points $H_1 , \dots , H_{n+1}$ in $\prj(V_k)$ determines $n+1$ divisors in $\prj(V^*_k)$, and that
\begin{equation*}
\{ (H_1 , \dots , H_{n+1}) \in \prj(V_k) \times \cdots \times \prj(V_k)  \mid H_1 \cap \cdots \cap H_{n+1} \cap \iota (X) \neq \emptyset \text{ in } \prj (V_k^*) \}
\end{equation*}
is a divisor in $\prj(V_k) \times \cdots \times \prj(V_k) $. The polynomial $ \Phi_{X,k} \in ( \mathrm{Sym}^{d_k} (V^*_k))^{\otimes (n+1)}$ defining this divisor, or the point $[ \Phi_{X, k} ]$ in $\prj( (\mathrm{Sym}^{d_k} (V^*_k))^{\otimes (n+1)})$ is called the \textbf{Chow form} of $X \inj \prj (H^0 (X,L^k)^*)$. It is a classical fact \cite{kollar,mumford77} that $[ \Phi_{X,k} ]$ corresponds bijectively to a subvariety in $\prj (H^0 (X,L^k)^*)$ of dimension $n$ and degree $d_k$. 

Chow stability of $(X,L)$ is nothing but the GIT stability of the point $[ \Phi_{X,k} ] \in \prj(( \mathrm{Sym}^{d_k} (V^*_k))^{\otimes (n+1)})$ with respect to the $SL(V_k^*)$-action on $ (\mathrm{Sym}^{d_k} (V^*_k))^{\otimes (n+1)}$. More precisely, it can be defined as follows.

\begin{definition} \label{defchstpsssus}
A polarised \kah manifold $(X,L)$ is said to be:
\begin{enumerate}
\item \textbf{Chow polystable at the level $k$} if the $SL(V_k^*)$-orbit of $\Phi_{X,k}$ is closed in $ (\mathrm{Sym}^{d_k} (V^*_k))^{\otimes (n+1)}$,
\item \textbf{Chow stable at the level $k$} if it is Chow polystable and $\Phi_{X,k}$ has finite isotropy,
\item \textbf{Chow semistable at the level $k$} if the $SL(V_k^*)$-orbit of $\Phi_{X,k}$ does not contain $0 \in (\mathrm{Sym}^{d_k} (V^*_k))^{\otimes (n+1)}$,
\item \textbf{Chow unstable at the level $k$} if it is not Chow semistable,
\item \textbf{asymptotically Chow stable} (resp.~\textbf{polystable}, \textbf{semistable}) if there exists $k_0 \in \mathbb{N}$ such that it is Chow stable (resp.~polystable, semistable) at the level $k$ for all $k \ge k_0$.
\end{enumerate}
\end{definition}

We recall the following fundamental theorem.

\begin{theorem} \emph{(Luo \cite{luo}, Zhang \cite{zhang})} \label{lzcsbalmauttriv}
	Suppose that $\mathrm{Aut}_0 (X,L)$ is trivial. Then $(X,L)$ is Chow stable at the level $k$ if and only if there exists $H \in \mathcal{B}_k$ such that $\rho_k (\omega_{H}) = \text{const}$. Such a \kah metric $\omega_H$ is called a \textbf{balanced} metric.
\end{theorem}

%


\subsection{Chow polystability relative to a torus} \label{scpsreltt}
We now review the version of Chow stability which is ``relative'' to the automorphism group $G= \mathrm{Aut}_0 (X,L)$, as introduced by Mabuchi \cite{mab04ext}. The reader is referred to the survey given in  Apostolov--Huang \cite{ah} for further discussions. Since we have $\theta$ as in Lemma \ref{lemdefofthtosl}, choosing a real torus $T$ in $K$, we can consider the representation $\theta |_{T^c} : T^c \actson H^0(X,L^k)$ where $T^c$ is the complexification of $T$. We then consider a subspace
\begin{equation*}
V_k (\chi) := \{ s \in H^0 (X,L^k) \mid \theta (t) \cdot s = \chi(t) s \text{ for all } t \in T^c \}
\end{equation*}
of $H^0 (X,L^k)$, where $\chi \in \mathrm{Hom} (T^c , \cx^*)$ is a character. We then have a decomposition
\begin{equation} \label{dech0bchoft}
H^0 (X,L^k) = \bigoplus_{\nu =1}^r V_k (\chi_{\nu})
\end{equation}
for mutually distinct characters $\chi_1 , \dots , \chi_r \in \mathrm{Hom} (T^c , \cx^*)$. We define 
\begin{equation*}
G^c_T :=  \left\{ \mathrm{diag} (A_1 , \dots , A_r) \in \prod_{\nu =1}^r GL(V_k (\chi_{\nu})) \left|  \prod_{\nu=1}^r \det(A_{\nu}) =1  \right. \right\}     
\end{equation*}
for the ``elements in $SL(H^0 (X,L^k))$ that commute with the $T^c$-action'', and
\begin{equation*}
G^c_{T^{\perp}} := \left\{ \mathrm{diag} (A_1 , \dots , A_r) \in \prod_{\nu =1}^r GL(V_k (\chi_{\nu})) \left|  \prod_{\nu=1}^r \det(A_{\nu})^{1 + \log |\chi_{\nu} (t)|} =1 \text{ for all } t \in T^c  \right. \right\} 
\end{equation*}
for the ``subgroup of $G^c_T$ that is orthogonal to the $T^c$-action''. 

\begin{remark}
	Write $\mathfrak{g}$ for $\theta_* (\mathfrak{aut}(X,L)) $ which is a Lie subalgebra of $ \mathfrak{sl} (H^0(X,L^k))$. Following \cite[\S 1.3]{szethesis}, we define
	\begin{equation*}
		\mathfrak{g}_T := \left\{ \alpha \in \mathfrak{g} \mid [\alpha , \beta ] = 0 \text{ for all } \beta \in \mathrm{Lie}T^{c} \right\}
	\end{equation*}
	where $[,]$ is the commutator, and
	\begin{equation*}
		\mathfrak{g}_{T^{\perp}} := \left\{ \alpha \in \mathfrak{g}_T \mid \langle \alpha , \beta \rangle = 0 \text{ for all } \beta \in \mathrm{Lie}T^{c} \right\}
	\end{equation*}
	where $\langle , \rangle$ is the natural inner product on $\mathfrak{sl} (H^0(X,L^k))$.
	
	Direct computation shows that $G^c_T$ is the connected Lie group corresponding to $\mathfrak{g}_T$, and $G^c_{T^{\perp}}$ to $\mathfrak{g}_{T^{\perp}}$.
\end{remark}

We now define the relative Chow stability as follows.

\begin{definition} \label{defwcpsalkrelttpcs}
A polarised \kah manifold $(X,L)$ is said to be \textbf{Chow polystable at the level $k$ relative to $T$} if the $G^c_{T^{\perp}}$-orbit of $\Phi_{X,k}$ is closed in $ (\mathrm{Sym}^{d_k} (V^*_k))^{\otimes (n+1)}$.
\end{definition}

On the other hand, we can consider an action of a smaller group $\widetilde{G_{T^{\perp}}^c}: = \prod_{\nu =1}^r SL(V_k (\chi_{\nu}))$; observe $\widetilde{G_{T^{\perp}}^c} \le G^c_{T^{\perp}}$. This leads to the notion of ``weak'' relative Chow polystability as follows (cf.~\cite{mab04ext, ah}).

\begin{definition} \label{defwcpsalkreltt}
A polarised \kah manifold $(X,L)$ is said to be \textbf{weakly Chow polystable at the level $k$ relative to $T$} if the $\widetilde{G_{T^{\perp}}^c}$-orbit of $\Phi_{X,k}$ is closed in $ (\mathrm{Sym}^{d_k} (V^*_k))^{\otimes (n+1)}$.
\end{definition}

In the case $\mathrm{Aut}_0 (X,L)$ is trivial, Chow stability corresponds to the existence of balanced metrics, as proved by Luo \cite{luo} and Zhang \cite{zhang} (cf.~Theorem \ref{lzcsbalmauttriv}). The notion of ``balanced'' metrics in the relative setting was proposed by Mabuchi \cite{mab04ext} as follows.

\begin{definition}
A hermitian metric $h \in \mathcal{H} (X,L)$ is said to be \textbf{balanced at the level $k$ relative to $T$} if $h$ is $T$-invariant and satisfies the following property: writing $\{ s_{\nu, i} \}_{\nu , i}$ for a $Hilb(h)$-orthonormal basis for $H^0 (X,L^k)$, where each $\{ s_{\nu ,i} \}_{i}$ is a $Hilb(h)$-orthonormal basis for $V_k (\chi_{\nu})$, there exist positive constants $(b_1 , \dots , b_r)$, $b_{\nu} >0$, such that
\begin{equation*}
\sum_{\nu , i} b_{\nu} |s_{\nu ,i}|^2_{h^k} = 1.
\end{equation*}
\end{definition}

A fundamental theorem is the following.

\begin{theorem}\emph{(Mabuchi \cite{mab04ext, mab11}; see also \cite[Theorems 2 and 4]{ah})} \label{mabpbcprelt}
$(X,L)$ is Chow polystable at the level $k$ relative to $T$ if and only if it admits a hermitian metric balanced relative to $T$ with each $b_{\nu}$ satisfying
\begin{equation} \label{eqofbnuchstrelt}
b_{\nu} = 1 + \log |\chi_{\nu} (t)| 
\end{equation}
for some $t \in T^c$, i.e. $b_{\nu}$'s are the eigenvalues of $I+ \xi$ for some $ \xi \in \theta_* (\mathrm{Lie}(T^c))$.
\end{theorem}

\begin{corollary} \label{cahrelmmwcst}
\emph{(cf.~\cite[\S 2]{ah})}
$(X,L)$ is Chow polystable at the level $k$ relative to $T$ if and only if there exists a $T$-invariant basis $\underline{s}$ for $H^0 (X, L^k)$ such that 
\begin{equation*}
	\bar{\mu}_X (\underline{s}) = \frac{Vk^n}{N} I +  \xi
\end{equation*}
for some $\xi \in \theta_* (\ai \mathrm{Lie}(T))$. In other words, the trace free part of $\bar{\mu}_X (\underline{s})$ generates a holomorphic automorphism of $\prj^{N-1}$ which preserves the image of $X$ under the Kodaira embedding.
\end{corollary}

\begin{proof}
	Suppose that we have a metric balanced at the level $k$ relative to $T$, satisfying $\sum_{\nu , i} b_{\nu} |s_{\nu ,i}|^2_{h^k} = 1$ with $b_{\nu}$'s satisfying (\ref{eqofbnuchstrelt}). We then see that $h$ can be written as $h = FS(H)$ with $H$ having $\underline{s'} = \{ \sqrt{b_{\nu} } s_{\nu, i} \}_{ \nu, i}$ as its orthonormal basis (cf.~equation (\ref{defoffseq})), and that $H$ is $T$-invariant (cf.~Definition 1 of \cite{ah} and the argument that follows; see also \cite[\S 2.3]{yhpreprint}). Then, the centre of mass $\bar{\mu}_X (\underline{s'})$ with respect to this basis can be computed as
\begin{align*}
\bar{\mu}_X (\underline{s'}) &= \frac{Vk^n}{N} I +  \frac{Vk^n}{N} \mathrm{diag} (\log |\chi_{1} (t)| \mathrm{id}_{V_k (\chi_1)}, \dots , \log |\chi_{r} (t)| \mathrm{id}_{V_k (\chi_r)}) \\
&=  \frac{Vk^n}{N} I + \frac{Vk^n}{N} \log \theta(t),
\end{align*}
and we simply define $\xi : = \frac{Vk^n}{N} \log \theta (t) \in \theta_* (\ai \mathrm{Lie}(T))$. 

Conversely, writing $A = \frac{Vk^n}{N} \log \theta (t)$ for some $t \in T^c / T$, suppose that we have a $T$-invariant basis $\underline{s}'$ such that $\bar{\mu}_X (\underline{s}') = \frac{Vk^n}{N} I +  \frac{Vk^n}{N} \log \theta(t)$. Diagonalising $\log \theta (t)$, and defining $b_{\nu}$'s as in (\ref{eqofbnuchstrelt}), we see that $\{ \sqrt{b^{-1}_{\nu}} s'_{\nu, i} \}_{\nu , i}$ is a $Hilb(h)$-orthonormal basis, when $\{ s'_{\nu, i} \}_{\nu , i}$ is an $H$-orthonormal basis. We thus get
\begin{equation*}
1= \sum_{\nu , i}  | s'_{\nu ,i}|^2_{h^k} = \sum_{\nu , i}  b_{\nu} \left| \sqrt{b^{-1}_{\nu}}s'_{\nu ,i} \right|^2_{h^k}
\end{equation*}
as required, for $h = FS(H)$, which is $T$-invariant by \cite[\S 2.3]{yhpreprint}.
\end{proof}

We now recall the following ``weak'' version of the preceding Theorem \ref{mabpbcprelt}.

\begin{theorem} \label{mabpbwcst}
\emph{(Mabuchi \cite{mab04ext, mab11}; see also the discussion preceding Definition 5 of \cite{ah})}
$(X,L)$ is weakly Chow polystable at the level $k$ relative to $T$ if and only if it admits a hermitian metric balanced relative to $T$ with some $b_{\nu} >0$, not necessarily satisfying (\ref{eqofbnuchstrelt}). 
\end{theorem}

\begin{corollary}
	$(X,L)$ is weakly Chow polystable at the level $k$ relative to $T$ if and only if there exists a $T$-invariant basis $\underline{s}$ such that
\begin{equation*}
	\bar{\mu}_X (\underline{s}) = \mathrm{diag} (b_1 \mathrm{id}_{V_k (\chi_1)}, \dots , b_r \mathrm{id}_{V_k (\chi_r)})
\end{equation*}
with respect to the decomposition $H^0 (X,L^k) = \bigoplus_{\nu =1}^r V_k (\chi_{\nu})$, for some $b_{\nu} >0$ (not necessarily satisfying (\ref{eqofbnuchstrelt})).
\end{corollary}

In particular, Chow polystability relative to $T$ implies weak Chow polystability relative to $T$.

\begin{remark} \label{rembalstab1dir}
	It is important to note that the notion of relative Chow stability comes with certain parameters associated to the automorphism, and this implies that the weight $\{ b_{\nu} \}_{\nu}$ is \textit{a priori} not uniquely determined by the assumption that $(X,L^k)$ is (weakly) relatively Chow stable. On the other hand, when we construct relative balanced metrics as in \cite{yhpreprint,mab2016,santip17,seyrel}, the weight $\left\{ b_{\nu} \right\}_{\nu}$ is of some specific value; in particular, construction of relative balanced metrics is in general stronger than proving (weak) relative Chow stability, in the sense that they provide specific values of the weight $\left\{ b_{\nu} \right\}_{\nu}$.
\end{remark}

\begin{remark}
	In fact, Theorems \ref{lzcsbalmauttriv}, \ref{mabpbcprelt}, and \ref{mabpbwcst} can be proved by formulating (relative) Chow stability in terms of test configurations and using explicit formulae of the modified balancing energy. The details of this may appear elsewhere.
\end{remark}

\subsection{Proof of Theorem \ref{wrelchowstab}}
To prove Theorem \ref{wrelchowstab}, it suffices to establish the following.

\begin{proposition} \label{rbalippbbnnla}
	If there exists $H \in \mathcal{B}_k$ such that $\omega_H$ is $K$-invariant and satisfies $\bar{\partial} \mathrm{grad}^{1,0}_{\omega_H} \rho_k (\omega_H) = 0$, then $FS(H)$ is balanced at the level $k$ relative to the centre $Z$ of $K$ for some $b_{\nu} >0$.
\end{proposition}

\begin{proof}
	Recalling Remark \ref{remdefcomitofs}, Proposition \ref{invmuholo}, and the equation (\ref{invmuholoxi}), when we write $\{ s_i \}_i$ for an $H$-orthonormal basis, we see that the basis $\left\{ s'_i \right\}_i$ defined by
\begin{equation} \label{relchchbasanui}
	s'_i := k^{-n/2} \left( cI + \xi \right)^{1/2}_{ij} s_j ,
\end{equation}
is a $\int_X h^k_{FS(H)}(,)\frac{\omega^n_{H}}{n!}$-orthonormal basis, where $\left( cI + \xi \right)_{ij} $ is the matrix for $cI + \xi$ represented with respect to $\{ s_i \}_i$, which is positive definite by Remark \ref{invmuholoxirem}. Moreover, by replacing $\{ s_i \}_i$ by an $H$-unitarily equivalent basis if necessary, we may assume that $\xi$ is diagonal. For notational convenience, we write $\{s_{\nu , i} \}_{\nu , i} $ for $\{ s_i \}_i$ (resp.~$\{s'_{\nu , i} \}_{\nu , i} $ for $\{ s'_i \}_i$) for the rest of the proof, according to the decomposition (\ref{dech0bchoft}), just to make explicit which sector $V_k(\chi_{\nu})$ each basis element $s_i$ belongs to.

Recall that Theorem \ref{mtqekm} implies $\xi \in \theta_* (\ai \mathfrak{z})$, and hence we may write
\begin{equation*}
\xi_{ij} = \mathrm{diag} (a_1 \mathrm{id}_{V_k (\chi_1)}, \dots , a_r \mathrm{id}_{V_k (\chi_r)}),
\end{equation*}
with respect to the characters $\chi_1 , \dots , \chi_r$ of $Z^c$. Thus we can write
\begin{equation*}
\left( cI + \xi \right)_{ij} = \mathrm{diag} (b_1^{-1} \mathrm{id}_{V_k (\chi_1)}, \dots , b_r^{-1} \mathrm{id}_{V_k (\chi_r)})
\end{equation*}
for some $b_{\nu}>0$, by recalling that $cI + \xi$ is positive definite (cf.~Remark \ref{invmuholoxirem}). In particular, (\ref{relchchbasanui}) can be re-written as $s'_{i , \nu} = k^{-n/2} b_{\nu}^{-1/2} s_{i , \nu} $. This means that we can write
\begin{equation} \label{polbrbaleqbnu}
	\sum_{\nu , i} b_{\nu} |s'_{\nu , i}|^2_{FS(H)^k} = k^{-n} \sum_{\nu , i}  |s_{\nu,i}|^2_{FS(H)^k} = \cst
\end{equation}
by the equation (\ref{defoffseq}), as required. Observe also that these $b_{\nu}$'s in the above equation are the eigenvalues of $(cI + \xi)^{-1}$, and \textit{not} of $cI+ \xi$, so a priori does not satisfy the equation (\ref{eqofbnuchstrelt}).
\end{proof}

\begin{remark} \label{rempbrbevalinva}
	The proof above in fact shows that $\omega_H$ satisfies $ \bar{\partial} \mathrm{grad}^{1,0}_{\omega_H} \rho_k (\omega_H) =0$ if and only if it satisfies the equation (\ref{polbrbaleqbnu}) with $b_{\nu}$'s being the eigenvalues of $(cI+\xi)^{-1}$ for some $\xi \in \theta_* (\ai \mathfrak{z})$; note the difference to the statement in Theorem \ref{mabpbcprelt}.
\end{remark}

\begin{remark}
Recalling that $Z$ is contained in any maximal torus in $K$, we finally note that Chow polystability relative to the centre $Z$ is stronger than that relative to any maximal torus in $K$.
\end{remark}
%

\section{Relative $K$-semistability from the point of view of quantisation} \label{rlksspovquant}

\subsection{Relative $K$-semistability}

We first recall the notion of test configurations that are compatible with a torus action, as defined by Sz\'ekelyhidi \cite{szeext}.

\begin{definition}
	A \textbf{test configuration} for $(X,L)$ of exponent $r$ is a $\cx^*$-equivariant flat family $\pi : \mathcal{X} \to \cx$ together with a $\cx^*$-equivariant ample line bundle $\mathcal{L}$ on $\mathcal{X}$ such that $\pi^{-1} (1) \cong (X,L^r)$.
	
	$(\XX. \LL)$ is said to be \textbf{compatible} with a complex torus $T^c \le \mathrm{Aut}_0 (X,L)$ if there exists a torus action on $(\XX, \LL)$ which
	\begin{enumerate}
		\item preserves the fibres of $\pi : \XX \to \cx$,
		\item commutes with the defining $\cx^*$-action of $(\XX, \LL)$,
		\item restricts to $T^c$ on $\pi^{-1}(t) \cong (X,L^r)$ for all $t \neq 0$.
	\end{enumerate}

	$(\XX , \LL)$ is said to be \textbf{product} if $\XX \cong X \times \cx$, and \textbf{trivial} if $\XX \cong X \times \cx$ with trivial $\cx^*$-action on $\cx$.
\end{definition}

Suppose that we have a test configuration $(\mathcal{X} , \mathcal{L})$ generated by a $\cx^*$-action $\alpha$. Then there exists an embedding $\mathcal{X} \inj \prj (H^0 (X,L^r)^*)$ by \cite[Proposition 3.7]{rt2}, with a generator $A_r \in \mathfrak{sl} (H^0 (X,L^r))$ such that $\XX$ is equal to the closure of the $\cx^*$-orbit of $X \inj \prj (H^0 (X,L^r)^*))$ generated by $A_r$.

Given $(\XX , \LL)$, we can construct a sequence of test configurations $(\XX , \LL^{\otimes k})$ for $k \in \mathbb{N}$, and as above we can write this as a closure of the $\cx^*$-orbit of $X \inj \prj(H^0 (X,L^{rk})^*)$ generated by $A_{rk} \in \mathfrak{sl} (H^0 (X,L^{rk}))$, say. Note that $A_{rk}$ is also equal to the generator of the action $\alpha : \cx^* \actson H^0 (\XX_0, \LL^{k}|_{\XX_0})$ where $\XX_0 = \pi^{-1}(0)$.

By Riemann--Roch and equivariant Riemann--Roch, we write
\begin{align}
	\mathrm{dim} H^0 (X , L^{rk}) &= a_0 (rk)^{n} + a_1 (rk)^{n-1} + \cdots, \label{rrdimhzxlkexp} \\
	\mathrm{tr} (A_{rk}) &= b_0 (rk)^{n+1} + b_1 (rk)^n + \cdots. \label{eqrrtrarkexp}
\end{align}
Observe that $a_0$ is equal to the volume $V$.

\begin{definition} \label{defchwbtestcfg}
The \textbf{Chow weight} of $(\mathcal{X} , \mathcal{L})$ is defined by
\begin{equation*}
	\mathrm{Chow}_r (\mathcal{X} , \mathcal{L}) := r b_0 -  \frac{a_0 \tr (A_r)}{\mathrm{dim} H^0 (X,L^r)}.  
\end{equation*}
\end{definition}

\begin{remark}
It is well-known that $\mathrm{Chow}_{r} (\XX , \LL) > 0$ for all nontrivial test configurations is equivalent to Chow stability of $X \inj \prj (H^0(X,L^r)^*)$ as defined in Definition \ref{defchstpsssus} (cf.~\cite[Proposition 2.11]{mumford77}), although we will not need to use this fact in what follows.
\end{remark}

In what follows, we shall assume that $L$ is very ample, and take $r=1$ for notational convenience; this can be achieved by simply replacing $L$ by a large tensor power.

\begin{definition} \label{defofdfinvalg}
	The \textbf{Donaldson--Futaki invariant} $DF(\XX , \LL)$ is defined as 
	\begin{equation*}
	DF(\XX , \LL) = (a_1 b_0 - a_0 b_1)/a_0.
	\end{equation*}
\end{definition}

\begin{remark} \label{remlimchowdf}
	Note that, by using the expansions (\ref{rrdimhzxlkexp}) and (\ref{eqrrtrarkexp}), we have
	\begin{equation} \label{remlimchowdfeq}
		 DF (\XX , \LL) = \lim_{k \to \infty} \mathrm{Chow}_k (\XX , \LL^{\otimes k}).
	 \end{equation}
\end{remark}

Let $\beta_1 , \dots , \beta_d$ be a basis for the $\cx^*$-actions generating $T^c$, with generators $B_{1,k} , \dots , B_{d,k} \in \mathfrak{sl} (H^0 (X,L^{k}))$. We define an inner product $\langle \alpha , \beta_i \rangle$ for $i = 1 , \dots , d$  as the leading coefficient of the following asymptotic expansion
\begin{equation} \label{eqrrtrabkntwo}
	\mathrm{tr} \left( A_{k}  B_{i,k} \right)  = \langle \alpha , \beta_i \rangle k^{n+2} + O(k^{n+1}), 
\end{equation}
as defined by Sz\'ekelyhidi \cite{szeext}, by recalling the well-known equivariant Riemann--Roch theorem.

Then we define the Donaldson--Futaki invariant relative to $T^c$ as follows.

\begin{definition}
	\begin{equation} \label{defofreldft}
		DF_{T^c} (\XX , \LL) = DF(\alpha) - \sum_{i=1}^d \frac{\langle \alpha , \beta_i \rangle}{\langle \beta_i , \beta_i \rangle} DF(\beta_i)
	\end{equation}
	where $DF(\alpha) = DF(\XX , \LL)$, and $DF(\beta_i)$ stands for the Donaldson--Futaki invariant for the product test configuration generated by $\beta_i$.
\end{definition}

\begin{remark}
	Writing $\bar{\alpha}$ for the projection of $\alpha$ orthogonal to $T^{c}$ with respect to $\langle , \rangle$ defined in (\ref{eqrrtrabkntwo}), we have $DF_{T^{c}} (\XX , \LL) = DF(\bar{\alpha})$.
\end{remark}

Following \cite{szeext}, we define relative $K$-semistability.
\begin{definition}
	$(X,L)$ is said to be \textbf{$K$-semistable relative to $T^{c} \le \mathrm{Aut}_0 (X,L)$} if $DF_{T^{c}} (\XX , \LL) \ge 0$ for all test configurations $(\XX , \LL)$ compatible with $T^{c}$.
\end{definition}

When we consider relative $K$-\textit{poly}stability there is a subtlety concerning the triviality of test configurations, as noted in \cite{lx}. However the result we will aim for in this paper is about relative $K$-semistability, and hence we will not be concerned with this subtlety here.

\subsection{Proof of Theorem \ref{relkss}} \label{pfthmrelkss}

We consider the case when we take $T^c$ to be the torus generated by the extremal vector field. We write $\chi$ for the $\cx^*$-action generated by the extremal vector field, with the scaling given by $DF(\chi) = \langle \chi , \chi \rangle$. Then the definition (\ref{defofreldft}) of the relative Donaldson--Futaki invariant implies $DF_{\chi} (\XX , \LL) = DF (\alpha) - \langle \alpha , \chi \rangle$.

We can write $\chi$ explicitly in terms of the scalar curvature $S(\omega)$ of the extremal metric $\omega$ as follows. Write $v_s$ for the Hamiltonian vector field defined by $S(\omega)$, where we use $\iota_{v_s} \omega  = - d S(\omega)$ for the sign convention for the Hamiltonian vector field. Recall also the faithful group representation $\theta : \mathrm{Aut}_0 (X,L) \inj SL(H^0 (X,L^{k}))$ as in Lemma \ref{lemdefofthtosl}. Write $B_{\chi , k} \in \mathfrak{sl} (H^0 (X, L^{k}))$ for the generator of $\chi$ which we may assume is hermitian (cf.~\cite{donlb}). Thus the generator of the $\cx^*$-action on $H^0 (X,L^{k})$ defined by the product test configuration generated by $v_s$ is a constant multiple of $\theta_*(Jv_s) = \ai \theta_* (v_s)$, where $J$ is the complex structure of $X$. Note that we have
\begin{equation} \label{eqjvsclftssq}
	DF \left(  \frac{\theta_* (J v_s)}{2 \pi } \right) =  \frac{1}{4 \pi} \int_X (S(\omega) - \bar{S})^2 \frac{\omega^n}{n!},
\end{equation}
which will be proved in \S \ref{expfflrrsn} (cf.~Corollary \ref{calgdfancfssc}), together with an explicit density formula for the equivariant Riemann--Roch theorem (Theorem \ref{exeploceqrrt}). Thus we look for a constant $C$ such that $B_{\chi , k} =   \frac{C}{2 \pi} \theta_* (J v_s)$. $C$ can be determined by the scaling $DF(\chi) = \langle \chi , \chi \rangle$. Now (\ref{eqjvsclftssq}) implies
\begin{equation*}
	DF(\chi) = DF \left(  C \frac{\theta_* (J v_s)}{2 \pi } \right) =  \frac{C}{4 \pi} \int_X (S(\omega) - \bar{S})^2 \frac{\omega^n}{n!},
\end{equation*}
and on the other hand we have
\begin{align*}
	\langle \chi , \chi \rangle &= \lim_{k \to \infty} \frac{C^2}{k^{n+2}}  \mathrm{tr} \left( \left(  \frac{\theta_* (J v_s)}{2 \pi } \right)^2 \right)   \\
	&= C^2 \int_X (S(\omega) - \bar{S})^2 \frac{\omega^n}{n!} ,
\end{align*}
by equivariant Riemann--Roch \cite[Proposition 7.16]{sze}. Thus $DF(\chi) = \langle \chi , \chi \rangle$ implies $C =  1/4 \pi $, and hence we get
\begin{equation*}
	B_{\chi, k} =  \frac{\theta_* (J v_s)}{8 \pi^2 }.
\end{equation*}

Suppose that we have a $\cx^*$-action $\beta$ generating a test configuration $(\XX_{\beta} , \LL_{\beta})$, and let $B \in \mathfrak{sl} (H^0 (X,L))$ be its generator. Writing $B_{k} \in \mathfrak{sl} (H^0 (X , L^{k}))$ for the generator of the $\cx^*$-action $\beta : \cx \actson H^0 (X, L^k)$, we have 
\begin{align}
	\langle \beta , \chi \rangle &= \lim_{k \to \infty} k^{-n-2} \mathrm{tr} \left( B_{k} \left( \frac{\theta_* (J v_s)}{8 \pi^2} \right) \right) \notag \\
	&= \frac{1}{ 8\pi^2} \lim_{k \to \infty} k^{-n-2} \mathrm{tr} \left( B_{k} \theta_* (Jv_s) \right). \label{ipbtchlimmatrix}
\end{align}

Since the modified balancing energy $\mathcal{Z}^A$ (cf.~Definition \ref{defmodbalenza} and (\ref{fmlindzmmbct})) is a geodesically convex function which admits a critical point, we have
\begin{align*}
	&\lim_{t \to \infty} \frac{d}{dt} \mathcal{Z}^A (H_t) \\
	&= \lim_{t \to \infty} \mathrm{tr} (B_k \bar{\mu}_X (H_t)) - \frac{k^n V}{N} \mathrm{tr}\left( B_k \left( I + C_A I + \frac{A}{2 \pi k} \right)^{-1} \right) > 0 ,
\end{align*}
for all geodesics $\left\{ H_t = e^{-B_k t} H_0 \right\} \subset \mathcal{B}_{k}$, where $B_k  \in \mathfrak{sl} (H^0(X,L^{k}))$ is hermitian and commutes with $\chi$. By recalling Definition \ref{defchwbtestcfg} and noting that $B_k$ is trace-free, we get
\begin{equation*}
	 \lim_{t \to \infty} \mathrm{tr} (B_k  \bar{\mu}_X (H_t)) = k^n \mathrm{Chow}_{k} (\XX_{\beta}, \LL^{\otimes k}_{\beta})
\end{equation*}
from \cite[Proposition 3]{donlb}. Recalling also (\ref{remlimchowdfeq}), we get
\begin{align}
	&\lim_{k \to \infty} \lim_{t \to \infty} k^{-n} \frac{d}{dt} \mathcal{Z}^A (H_t) \notag \\
	&= DF(\XX_{\beta}, \LL_{\beta}) - \lim_{k \to \infty} \frac{V}{N} \mathrm{tr} \left( B_k \left( I + C_A I + \frac{A}{ 2 \pi k} \right)^{-1} \right)  \ge 0. \label{rdflmznn}
\end{align}

We evaluate the second term of the above inequality, and show that it is equal to the correction term $\langle \beta, \chi \rangle$ in the relative Donaldson Futaki invariant $DT_{\chi} (\XX_{\beta} , \LL_{\beta})$.

Now the well-known expansion of the Bergman function \cite{bouche,lu,ruan,tian90,yauberg,zelditch} and $\omega_k \to \omega$ (as $k \to \infty$, cf.~Theorem \ref{mtqekm}) implies
\begin{align}
	\frac{V}{N} \mathrm{grad}_{\omega_k} \rho_k (\omega_k) &= \mathrm{grad}_{\omega} \left( 1 + \frac{1}{4 \pi k} (S(\omega) - \bar{S}) + O(k^{-2}) \right) \notag \\
	&=\frac{1}{4 \pi k} \mathrm{grad}_{\omega} S(\omega) + O(k^{-2}) \notag \\
	&= - \frac{1 }{4 \pi k } Jv_s + O(k^{-2}). \notag
\end{align}

Hence we have
\begin{equation*}
	\frac{A}{2 \pi k} = - \frac{\theta_*(Jv_s)}{8 \pi^2 k^2} + \text{higher order terms in } k^{-1}.
\end{equation*}

By further noting that $|| A ||_{op}$ is bounded uniformly of $k$ and $C_A = O(k^{-1})$ (cf.~Theorem \ref{mtqekm}), and also recalling $\mathrm{tr} (B_k) =0$, we get
\begin{align*}
	&\lim_{k \to \infty} \frac{V}{N} \mathrm{tr} \left( B_k \left( I + C_A I + \frac{A}{ 2 \pi k} \right)^{-1} \right) \\
	&= \lim_{k \to \infty} \frac{V}{N}\mathrm{tr} \left( B_k \left( I - C_A I - \frac{A}{2 \pi k} +  \text{higher order terms in } k^{-1}  \right) \right) \\
	&= \lim_{k \to \infty} \frac{V}{N}\mathrm{tr} \left( B_k \left( I - C_A I - \frac{A}{2 \pi k}  \right) \right) \\
	&= \lim_{k \to \infty} \frac{V}{N}\mathrm{tr} \left( B_k \frac{1}{8 \pi^2 k^2} \theta_*(J v_s) \right) \\
	&= \lim_{k \to \infty} \frac{1 }{k^{n+2}} \frac{1 }{8 \pi^2} \mathrm{tr} \left( B_k \theta_* (Jv_s) \right) \\
	&=\langle \beta , \chi \rangle .
\end{align*}

Thus, (\ref{rdflmznn}) can be written as
\begin{equation*}
	\lim_{k \to \infty} \lim_{t \to \infty} k^{-n} \frac{d}{dt} \mathcal{Z}^A (H_t)  = DF(\XX_{\beta}, \LL_{\beta}) - \langle \beta , \chi \rangle  \ge 0 .
\end{equation*}
Since this inequality holds for any $\cx^*$-action $\beta$ that commutes with $\chi$, we finally get
\begin{equation*}
	DF_{\chi} (\XX_{\beta} , \LL_{\beta})= DF(\XX_{\beta}, \LL_{\beta}) - \langle \beta , \chi \rangle  \ge 0 ,
\end{equation*}
for any test configuration $(\XX_{\beta} , \LL_{\beta})$, as required. This completes the proof of Theorem \ref{relkss}.

\section{Explicit formula for the local equivariant Riemann--Roch theorem} \label{expfflrrsn}
When the test configuration is product, it is well-known that the equivariant Riemann--Roch theorem admits a differential-geometric formula, such that the coefficients in the expansion (\ref{eqrrtrarkexp}) can be computed from curvature quantities. We shall prove in this section an explicit formula for the local density function for this, which reduces the expansion (\ref{eqrrtrarkexp}) to the one of the Bergman function.

Let $v \in \mathrm{Lie}(K)$ be a real holomorphic Hamiltonian vector field, satisfying 
\begin{equation*}
	\iota(v) \omega_h = - d \psi
\end{equation*}
where $\omega_h$ is a $K$-invariant \kah metric. Writing $\bar{\rho}_k (\omega_h) := \frac{V}{N} \rho_k (\omega_h)$ for the re-scaled Bergman function of $\omega_h$, we state the main result of this section as follows.
\begin{theorem}
	\label{exeploceqrrt}
Writing $A := \theta_* (J v)$ by using $\theta$ in Lemma \ref{lemdefofthtosl}, we have
\begin{equation*}
	\frac{V}{k N} \mathrm{tr} \left( \frac{\theta_* (Jv)}{2 \pi} \right) = - \int_X \psi  \bar{\rho}_k (\omega_h) \frac{\omega^n_h}{n!} - \int_X  \frac{1}{ 4 \pi k} \left( d \psi , d  \bar{\rho}_k (\omega_h) \right)_{\omega_h} \frac{\omega^n_h}{n!}.
\end{equation*}
\end{theorem}
By using the asymptotic expansion for the Bergman function and recalling Definition \ref{defofdfinvalg}, we obtain the following result.
\begin{corollary} \label{calgdfancfssc}
The Donaldson--Futaki invariant for the product test configuration generated by $\theta_* (Jv) /2 \pi$ admits a differential-geometric formula as follows.
\begin{equation} \label{algdfancfssc}
	DF \left( \frac{\theta_* (Jv)}{2 \pi} \right) = \frac{1}{4\pi}  \int_X \psi (S(\omega_h )- \bar{S}) \frac{\omega^n_h}{n!} .
\end{equation}
\end{corollary}
The connection between the algebraically defined Donaldson--Futaki invariant and the analytically defined Futaki invariant (on the right hand side) is a well-known theorem in \kah geometry \cite{dontoric}, but the the above formula explicitly specifies the generator of the test configuration, including the sign and the scaling, in terms of the vector field $v$. By choosing $v$ to be the extremal vector field $v_s$, we get the formula (\ref{eqjvsclftssq}).

\begin{remark}
	Sz\'ekelyhidi \cite[\S 7.3]{sze} introduced the \textbf{$S^1$-equivariant Bergman kernel} $B^{S^1}_{h^k}$ as a local density function for  $\mathrm{tr} \left( \theta_* (v) / 2 \pi \ai \right)$ (cf.~(\ref{defsoeqbksz})). The proof of Theorem \ref{exeploceqrrt} is based on the following explicit formula for $B^{S^1}_{h^k}$ as
	\begin{equation} \label{sobkelasyexp}
		B^{S^1}_{h^k} = k \left( \psi \bar{\rho} (\omega_h)  + \frac{1}{ 4 \pi k} \left( d \psi , d  \bar{\rho}_k (\omega_h) \right)_{\omega_h} \right).
	\end{equation}
	This formula enables us to obtain the full asymptotic expansion of $B^{S^1}_{h^k}$ in terms of the one of $\bar{\rho}_k (\omega_h)$, complementing the result given in \cite[Proposition 7.12]{sze}, which identifies the first two coefficients of the asymptotic expansion.
\end{remark}

\begin{proof}

Suppose that we take $H' := Hilb(h)$, which is $K$-invariant if $\omega_h$ is, by \cite[Lemma 2.25]{yhpreprint}. Then the Hamiltonian $\psi'$ with respect to $\omega_{FS(H')}$ can be written as
\begin{equation} \label{expfmofpsip}
	\psi' = -\frac{1}{2 \pi k} \sum_{i,j} A_{ij} h^k_{FS(H')} (s'_i , s'_j)
\end{equation}
where $\left\{ s'_i \right\}_i$ is a $H'$-orthonormal basis, by \cite[Lemma 4.3]{yhpreprint}. Writing $\bar{\rho}_k (\omega_h) := \frac{V}{N} \rho_k (\omega_h)$ for the re-scaled Bergman function, we have the well-known formula of Rawnsley \cite{rawnsley}
\begin{equation} \label{rawnsleyeq}
h^k_{FS(H')} = \bar{\rho} (\omega_h)^{-1} h^k ,
\end{equation}
which implies
\begin{align*}
	\iota(v) \omega_{H'} &= \iota(v) \left( \omega_h + \frac{\ai}{2 \pi k} \ddbar \log \bar{\rho}_k (\omega_h) \right) \\
	&= - d \left( \psi + \frac{1}{4 \pi k} \left( d \psi,  d \log \bar{\rho}_k (\omega_h) \right)_{\omega_h} \right)
\end{align*}
where $\left( , \right)_{\omega_h}$ is the pointwise inner product on 1-forms defined by $\omega_h$, by recalling \cite[Lemma 3.3]{yhpreprint}. In particular, this implies
\begin{equation} \label{psippsidetaddc}
	\psi' = \psi + \frac{1}{4 \pi k} \left( d \psi,  \frac{d  \bar{\rho}_k (\omega_h)}{ \bar{\rho}_k (\omega_h)} \right)_{\omega_h}
\end{equation}
up to an additive constant (cf.~Remark \ref{remaddcsthampsi}). Recalling (\ref{expfmofpsip}) and (\ref{rawnsleyeq}) we have
\begin{equation} \label{putthmstmtabveq}
	-\frac{1}{2 \pi k} \sum_{i,j} A_{ij} h^k (s'_i , s'_j) = \psi  \bar{\rho}_k (\omega_h) + \frac{1}{ 4 \pi k} \left( d \psi , d  \bar{\rho}_k (\omega_h) \right)_{\omega_h}.
\end{equation}
Now our scaling and sign convention implies that the $S^1$-equivariant Bergman kernel $B^{S^1}_{h^k}$ of Sz\'ekelyhidi can be written as
\begin{equation} \label{defsoeqbksz}
	B^{S^1}_{h^k} = -  \frac{1}{2 \pi} \sum_{i,j} A_{ij} h^k (s'_i , s'_j), 
\end{equation}
whereby establishing (\ref{sobkelasyexp}). Integrating both sides of the equation (\ref{putthmstmtabveq}), we thus get
\begin{equation*}
	\frac{V}{k N} \mathrm{tr} \left( \frac{\theta_* (Jv)}{2 \pi} \right) = - \int_X \psi  \bar{\rho}_k (\omega_h) \frac{\omega^n_h}{n!} - \int_X  \frac{1}{ 4 \pi k} \left( d \psi , d  \bar{\rho}_k (\omega_h) \right)_{\omega_h} \frac{\omega^n_h}{n!},
\end{equation*}
as claimed.

\end{proof}

\begin{remark} \label{remaddcsthampsi}
	Hamiltonian functions are well-defined only up to a constant, but the Hamiltonian $\psi$ in the statement of Theorem \ref{exeploceqrrt} is the one that is uniquely determined by (\ref{expfmofpsip}) and (\ref{psippsidetaddc}).
	
	On the other hand, recall that the ambiguity in Hamiltonian has precisely to do with the linearisation of the Hamiltonian vector $v$; changing $\psi \mapsto \psi + c$, $c \in \rl$, is exactly the same as changing
	\begin{equation*}
\frac{\theta_* (Jv)}{2 \pi} \mapsto \frac{\theta_* (Jv)}{2 \pi} - ckI
	\end{equation*}
	which leaves the test configuration unchanged.
\end{remark}

\section{Related results} \label{conjrelttprevkres}

Recent development in the field \cite{yhpreprint,mab2016,santip17,seyrel} has provided several notions of ``quantised'' or ``relatively balanced'' metrics adapted to the extremal metrics. There are subtle, yet nontrivial differences among them; the reader is referred to \cite[\S 6]{yhpreprint} for the review.

In this paper we shall be concerned with stability results that they imply. An important result in this direction is the following.

\begin{theorem} \emph{(Mabuchi \cite{mab2016}, Seyyedali \cite{seyrel})} \label{conjwracoexmfd}
The existence of extremal metrics in $c_1 (L)$ implies asymptotic Chow polystability of $(X,L)$ relative to any maximal torus in $K$.
\end{theorem}
Mabuchi \cite{mab2016} in fact proved a stronger result of $(X,L)$ being asymptotically Chow polystable relative to the centre $Z$ of $K$.

Recalling Corollary \ref{cahrelmmwcst}, this amounts to showing, for all large enough $k$, the existence of basis $\underline{s}_k$ for $H^0(X,L^k)$ such that $\bar{\mu}_X (\underline{s}_k) \in \theta_* (\mathfrak{aut} (X,L))$, which we shall also abbreviate as $\bar{\mu}_X \in \mathfrak{aut}(X,L)$.

The notion of \textbf{$\sigma$-balanced metric} was introduced by Sano in \cite{sanotalk}, where a \kah metric $\omega_h$ is said to be $\sigma$-balanced if there exists $\sigma \in \mathrm{Aut}_0 (X,L)$ such that $\omega_{FS(Hilb(h))} = \sigma^* \omega_h $. Sano--Tipler \cite{santip17} further proved the existence of $\sigma$-balanced metrics for all large enough $k$, assuming the existence of extremal metrics on $(X,L)$. 

Thus we have three notions of ``quantised'' or ``relatively balanced'' metrics in the literature, each of which exists for all large enough $k$ when $(X,L)$ admits an extremal metric:
\begin{enumerate}
	\item $\bar{\partial} \mathrm{grad}^{1,0}_{\omega} \rho_k (\omega) = 0$,
	\item $\bar{\mu}_X \in \mathfrak{aut}(X,L)$,
	\item $\omega_{FS(Hilb(h))} = \sigma^* \omega_h $.
\end{enumerate}
As discussed in \cite[\S 6]{yhpreprint}, equivalence of these notions is a subtle open problem. In fact, an argument that is almost identical to the proof of Proposition \ref{rbalippbbnnla} shows that the existence of $\sigma$-balanced metrics implies that $(X,L)$ is Chow polystable relative to a torus in $K$. Thus, the relationship of the above notions and stability properties can be summarised as follows.



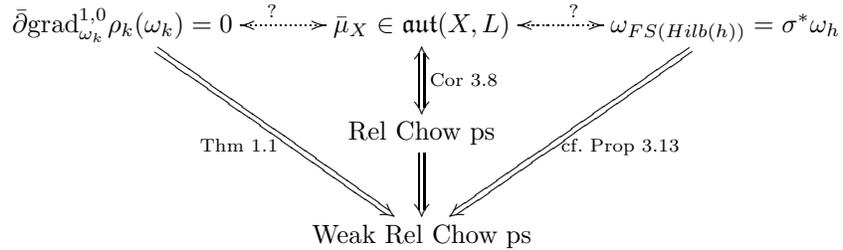
\begin{displaymath}
	\xymatrix{
	\bar{\partial} \mathrm{grad}^{1,0}_{\omega_k} \rho_k (\omega_k) = 0 \ar@{=>}[ddr]_{\text{Thm \ref{wrelchowstab}}}   & \bar{\mu}_X \in  \mathfrak{aut}(X,L) \ar@{<=>}[d]^{\text{Cor \ref{cahrelmmwcst}}} \ar@{<.>}[l]_{?} \ar@{<.>}[r]^{?} & \omega_{FS(Hilb(h))} = \sigma^* \omega_h \ar@{=>}[ddl]^{\text{cf.~Prop \ref{rbalippbbnnla}}}  \\
	&  \text{Rel Chow ps}  \ar@{=>}[d] \\
	& \text{Weak Rel Chow ps} &
	}
\end{displaymath}

When we assume that $(X,L)$ admits an extremal metric, we have three theorems establishing the existence of ``quantised extremal'' or ``relatively balanced'' metrics, as presented below (where ``A-'' stands for ``Asymptotic'').

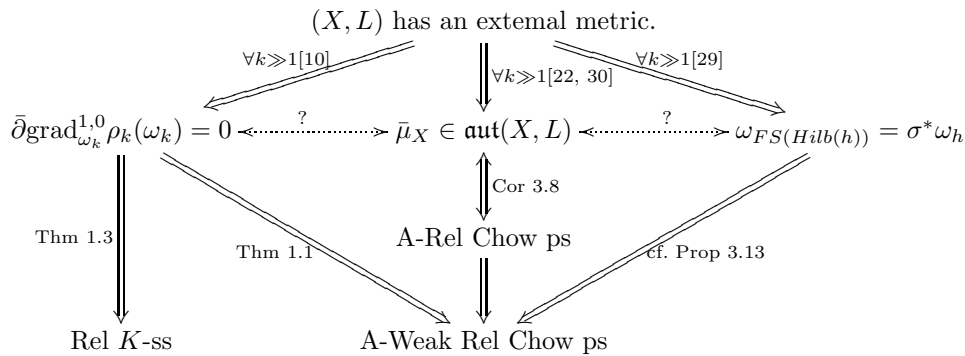
\begin{displaymath}
	\xymatrix{
	& (X,L) \text{ has an extemal metric.} \ar@{=>}[dl]_{\forall k \gg 1  \text{\cite{yhpreprint}}} \ar@{=>}[d]^{\forall k \gg 1  \text{\cite{mab2016,seyrel} }} \ar@{=>}[dr]^{\forall k \gg 1  \text{\cite{santip17} }}& \\
	\bar{\partial} \mathrm{grad}^{1,0}_{\omega_k} \rho_k (\omega_k) = 0 \ar@{=>}[ddr]_{\text{Thm \ref{wrelchowstab}}} \ar@{=>}[dd]_{\text{Thm \ref{relkss}}}   & \bar{\mu}_X \in  \mathfrak{aut}(X,L) \ar@{<=>}[d]^{\text{Cor \ref{cahrelmmwcst}}}\ar@{<.>}[l]_{?} \ar@{<.>}[r]^{?} & \omega_{FS(Hilb(h))} = \sigma^* \omega_h  \ar@{=>}[ddl]^{\text{cf.~Prop \ref{rbalippbbnnla}}}   \\
	&  \text{A-Rel Chow ps}  \ar@{=>}[d] \\
	\text{Rel $K$-ss} & \text{A-Weak Rel Chow ps} &
%
	}
\end{displaymath}

Equivalence of three ``quantised'' or ``relatively balanced'' metrics would be desirable, partly because it would simplify the implications to various stability notions as in the diagram above.

Finally, we remark that the relative $K$-stability of extremal manifolds was proved by Stoppa--Sz\'ekelyhidi \cite{stosze} by using the lower bound of the Calabi functional and a blowup argument.

\bibliographystyle{amsplain}
\bibliography{2017_22_stability}

\begin{flushleft}
Aix Marseille Universit\'e, CNRS, Centrale Marseille, \\
Institut de Math\'ematiques de Marseille, UMR 7373, \\
13453 Marseille, France. \\
Email: \verb|yoshinori.hashimoto@univ-amu.fr|
\end{flushleft}

\end{document}